\documentclass[reqno]{amsart}
\usepackage{amsfonts}
\usepackage{graphicx}
\usepackage{amscd}
\usepackage{amsmath}
\usepackage{amssymb}

\makeatletter
\@namedef{subjclassname@2010}{%
  \textup{2010} Mathematics Subject Classification}
\makeatother

\theoremstyle{plain}
\newtheorem*{acknowledgement}{Acknowledgement}

\newtheorem{corollary}{\bf Corollary}

\newtheorem{lemma}{\bf Lemma}

\newtheorem{theorem}{\bf Theorem}
\theoremstyle{definition}

\newcommand{\ric}{\mathring{Ric}}

\renewcommand{\div}{{\rm div}}

\numberwithin{equation}{section}

\title[Critical metrics of the volume functional]{Isoperimetric inequality and Weitzenb\"ock type formula for critical metrics of the volume}

\author{H. Baltazar}
\author{R. Di\'ogenes}
\author{E. Ribeiro Jr.}

\address[H. Baltazar]{Universidade Federal do Piau\'{i} - UFPI, Departamento de Matem\'{a}tica,  64049-550, Te\-re\-si\-na, Piau\'{\i}, Brazil.}
\email{halyson@ufpi.edu.br}

\address[R. Di\'ogenes]{UNILAB, Instituto de Ci\^encias Exatas e da Natureza, 62785-000, Acarape, Cear\'a, Brazil.}
\email{rafaeldiogenes@unilab.edu.br}

\address[E. Ribeiro Jr]{Universidade Federal do Cear\'a - UFC, Departamento  de Matem\'atica, Campus do Pici, Av. Humberto Monte, Bloco 914,
60455-760, Fortaleza, Cear\'a, Brazil.}
\email{ernani@mat.ufc.br}

\thanks{H. Baltazar was partially supported by CNPq/Brazil and FAPEPI/Brazil}
\thanks{E. Ribeiro Jr was partially supported by grants from CNPq/Brazil (Grant: 303091/2015-0), PRONEX-FUNCAP/CNPq/Brazil and CAPES/ Brazil - Finance Code 001}

\date{December 1, 2017}

\begin{document}

\newcommand{\spacing}[1]{\renewcommand{\baselinestretch}{#1}\large\normalsize}
\spacing{1.2}

\begin{abstract}
We provide an isoperimetric inequality for critical metrics of the volume functional with nonnegative scalar curvature on compact manifolds with boundary. In addition, we establish a Weitzenb\"ock type formula for cri\-ti\-cal metrics of the volume functional on four-dimensional manifolds. As an application, we obtain a classification result for such metrics.
\end{abstract}

\maketitle

\section{Introduction}
\label{intro}

A classical topic in Riemannian geometry is to find canonical metrics on a given manifold $M^n.$ A promising way for that purpose is that of critical metrics of the Riemannian functionals, as for instance, the total scalar curvature functional and the volume functional. Einstein and Hilbert proved that the critical points of the total scalar curvature functional restricted to the set of smooth Riemannian structures on $M^n$ of unitary volume are Einstein (cf. Theorem 4.21 in \cite{besse}), and this result stimulated several interesting works. In this spirit, Miao and Tam \cite{miaotam,miaotamTAMS}  studied variational properties of the volume functional constrained to the space of metrics of constant scalar curvature on a given compact manifold with boundary. Indeed, volume is one of the natural geometric quantities used to study geometrical and topological properties of a Riemannian manifold.

In order to make our approach more understandable, we need to recall some terminology. Let $(M^{n},\,g)$ be a connected compact Riemannian manifold with dimension $n$ at least three and smooth boundary $\partial M.$ According to \cite{BDR,BDRR,CEM,miaotam} and \cite{miaotamTAMS}, we say that $g$ is, for simplicity, a {\it Miao-Tam critical metric} if there is a nonnegative smooth function $f$ on $M^n$ such that $f^{-1}(0)=\partial M$ and satisfies the overdetermined-elliptic system
\begin{equation}\label{eqMiaoTam1}
\mathfrak{L}_{g}^{*}(f)=-(\Delta f)g+Hess\, f-f Ric=g,
\end{equation} where $\mathfrak{L}_{g}^{*}$ is the formal $L^{2}$-adjoint of the linearization of the scalar curvature ope\-ra\-tor $\mathfrak{L}_{g},$ which plays a fundamental role in problems related to prescribing the scalar curvature function. Moreover, $Ric,$ $\Delta$ and $Hess$ stand, respectively, for the Ricci tensor, the Laplacian operator and the Hessian form on $M^n.$ Further, we will consider $Hess=\nabla^{2}$ on functions.

It is proved in \cite{miaotam,miaotamTAMS} that these critical metrics arise as critical points of the volume functional on $M^n$ when restricted to the class of metrics $g$ with pres\-cribed constant scalar curvature such that $g_{|_{T \partial M}}=h$ for a pres\-cribed Riemannian metric $h$  on the boundary. Corvino, Eichmair and Miao \cite{CEM} studied the modified problem of finding stationary points for the volume functional on the space of metrics whose scalar curvature is equal to a given constant. In particular, they were able to prove a deformation result which suggests that the information of scalar curvature is not sufficient in giving volume comparison. The classification problem for critical me\-trics of the volume functional is important and relevant in understanding the influence of the scalar curvature in controlling the volume of a given manifold. For more details on such a subject, we refer the reader to \cite{baltazar,BDR,BDRR,CEM,Kim,miaotam,miaotamTAMS,yuan} and references therein.

Isoperimetric problems are classical objects of study in mathematics. Generally speaking, the isoperimetric inequality is a geometric inequality involving the surface area of a set and its volume. The isoperimetric inequality on the plane states that the length $L$ of a closed curve on $\mathbb{R}^{2}$ and the area $A$ of the planar region that it encloses must to satisfy
$$L^{2}\geq 4\pi A.$$
Moreover, the equality holds if and only if the curve is a circle. In $\mathbb{R}^{n},$ the classical isoperimetric inequality asserts that if $M\subset\mathbb{R}^{n}$ is a
compact domain with smooth boundary $\partial M,$ then
\begin{equation}
\frac{|\partial M|}{|\partial\mathbb{B}_{1}^{n}|}\geq \frac{{\rm Vol}(M)^{\frac{n-1}{n}}}{{\rm Vol}(\mathbb{B}_{1}^{n})^{\frac{n-1}{n}}},
\end{equation}
where $|\partial M|$ denotes the $(n-1)$-dimensional volume of $\partial M$ and ${\rm Vol}(M)$ is the volume of $M.$ For sufficiently smooth domains, the $n$-dimensional isoperimetric inequality is equivalent to the Sobolev inequality on $\Bbb{R}^n$ (cf. \cite{Ma}). We refer to \cite{Ros} for a general discussion on this topic. 

Inspired by a classical result obtained in  \cite{BGH} and \cite{Shen}, it has been shown by Batista, Ranieri and the last two named authors \cite{BDRR} that the boundary $\partial M$ of a compact three-dimensional oriented Miao-Tam critical metric $(M^3,\,g)$ with connected boundary and nonnegative scalar curvature must be a $2$-sphere whose area satisfies the inequality $$|\partial M|\leq\frac{4\pi}{C},$$ where $C$ is a positive constant. Moreover, the equality holds if and only if $M^{3}$ is isometric to a geodesic ball in $\Bbb{R}^3$ or $\Bbb{S}^3.$ This result also holds for negative scalar curvature, provided that the mean curvature of the boundary satisfies $H > 2,$ as was proven in \cite{BLF}; see also \cite{BS}. Another upper bound estimate for the area of the boundary was obtained by Corvino, Eichmair and Miao (cf. \cite{CEM}, Proposition 2.5).

In the spirit of these quoted results and stimulated by the isoperimetric problem, we provide a lower bound estimate for the area of the boundary of a compact manifold  satisfying (\ref{eqMiaoTam1}). More precisely, we have established the following result.

\begin{theorem}\label{thmMainA}
Let $(M^n,\,g,\,f)$ be a compact, oriented, Miao-Tam critical metric with connected boundary $\partial M.$ Then the area of the boundary $|\partial M|$ must satisfy
\begin{equation}
|\partial M|\geq\frac{n+2}{2n(n-1)^{2}}H^{3}C_{R},
\end{equation} where $H$ is the mean curvature of $\partial M$ with respect to the outward unit normal and $C_{R}$ is a positive constant given by
$$C_{R}=\int_{M}(R^{2}f^{3}+3nRf^{2}+2n^{2}f)dV_{g},$$ where $dV_{g}$ stands for the volume element of $M^n.$ Moreover, equality holds if and only if $M^n$ is isometric to a geodesic ball in a simply connected space form $\Bbb{R}^{n},$ $\Bbb{H}^{n}$ or $\Bbb{S}^{n}.$
\end{theorem}

It is important to recall that, since $f$ is nonnegative, $H=\frac{1}{|\nabla f|}$ is constant on $\partial M$ and hence, $\partial M$ is totally umbilical. For more details, see \cite[Eq. (3.3)]{BDRR} and \cite[Theorem 7]{miaotam}.

In order to justify our second main result it is important to recall a classical example of Miao-Tam critical metric built in a Euclidean ball in $\Bbb{R}^n$ (cf. Example 1 in \cite{BDRR}, see also \cite{miaotam}). Firstly, we consider the triple $(\mathbb{B}_{r}^{n},g_{0},f),$ where $(\mathbb{B}_{r}^{n},g_{0})$ denotes the Euclidean ball in $\mathbb{R}^{n}$ of radius $r$ with standard metric $g_{0}$ and potential function $f(x)=\frac{1}{2(n-1)}(r^{2}-|x|^{2}).$ Thus, it is not hard to verify that $(\mathbb{B}_{r}^{n},g_{0},f)$ is a Miao-Tam critical metric. Furthermore, using the Co-Area formula (cf. \cite{SY}) jointly with a suitable change of variable we obtain

\begin{eqnarray}\label{intf}
\int_{\mathbb{B}_{r}^{n}} f dx&=&\frac{1}{2(n-1)}\left(r^{2}{\rm Vol}(\mathbb{B}_{r}^{n})-\frac{r^{n+2}}{n+2}|\partial \mathbb{B}^{n}_{1}|\right)\nonumber\\&=&\frac{1}{n(n-1)(n+2)}r^{n+2}|\partial \mathbb{B}^{n}_{1}|,
\end{eqnarray} where $\partial \mathbb{B}^{n}_{1}$ stands for the boundary of the unit ball. Consequently, a straightforward computation gives
 
\begin{equation}
\label{4gh}
\frac{|\partial\mathbb{B}_{r}^{n}|}{{\rm Vol}(\mathbb{B}_{r}^{n})^{\frac{n-1}{n}}}=\left(\frac{(n+2)n^{n}H^{n+2}}{(n-1)^{n+1}}\int_{\mathbb{B}_{r}^{n}}fdx\right)^{\frac{1}{n}},
\end{equation} where $H$ is the mean curvature of $\partial\mathbb{B}_{r}^{n}$ with respect to the outward unit normal which is, in this case, given by $H=\frac{n-1}{r}.$

One question that naturally arises from Eq. (\ref{4gh}) is to establish the interplay between vo\-lu\-me and area of the boundary for a general critical metric of the vo\-lu\-me functional on a compact manifold with boundary. In this context, Corvino, Eichmair and Miao \cite{CEM} were able to show that the area of the boundary $\partial M$ of an n-dimensional Miao-tam cri\-ti\-cal me\-tric with zero scalar curvature must have an upper bound depending on the volume of $M^n$ as follows $${\rm Vol}(M)\geq \frac{\sqrt{(n-2)(n-1)}}{n}\Big(\int_{\partial M}R_{_{h}}d\sigma_{g}\Big)^{-\frac{1}{2}}|\partial M|^{\frac{3}{2}},$$ where $R_{_{h}}$ is the scalar curvature of $(\partial M,\,h).$

In the sequel, motivated by this discussion above as well as the isoperimetric problem, we shall use Theorem \ref{thmMainA} to obtain a lower bound estimate (depending on the volume of $M^n$) for the area of the boundary of a compact manifold $M^n$ with nonnegative scalar curvature satisfying (\ref{eqMiaoTam1}). More precisely, we have the following isoperimetric type inequality for Miao-Tam critical metrics with nonnegative scalar curvature.

\begin{theorem}\label{thmMainB}
Let $(M^n,\,g,\,f)$ be a compact, oriented, Miao-Tam critical metric with connected boundary $\partial M$ and nonnegative scalar curvature. Then we have:
\begin{equation}
|\partial M|\geq (C_{R,H})^{\frac{1}{n}}{\rm Vol}(M)^{\frac{n-1}{n}},
\end{equation}
where $C_{R,H}$ is a positive constant given by $C_{R,H}=\frac{(n+2)n^{n-2}}{2(n-1)^{n+1}}H^{n+2}C_{R},$ $H$ is the mean curvature of $\partial M$ with respect to the outward unit normal and $C_{R}$ is described in Theorem \ref{thmMainA}. Moreover, equality holds if and only if $M^n$ is isometric to a geodesic ball in $\mathbb{R}^{n}.$
\end{theorem}

Recently, mainly motivated by \cite{Lucas}, Batista et al. \cite{BDRR} obtained a B\"ochner type formula for three-dimensional Riemannian manifolds satisfying (\ref{eqMiaoTam1}) involving the traceless Ricci tensor and the Cotton tensor; see also \cite{BalRi2}. As a consequence of such a B\"ochner type formula, they obtained some classification results for critical metrics of the volume functional on three-dimensional compact manifolds with boundary. Before discussing the four-dimensional case, it is important to emphasize that
dimension four displays fascinating and peculiar features, for this reason very much attention has been given to this dimension; see, for instance \cite{besse,GS}, for more details on this specific dimension.

In \cite{derd1}, Derdzi\'nski showed that every oriented four-dimensional Einstein manifold $(M^{4},\,g)$ satisfies the Weitzenb\"ock formula 
$$\Delta |W^{\pm}|^{2}=2|\nabla W^{\pm}|^{2}+R|W^{\pm}|^{2}-36 \det W^{\pm},$$ where $W^{+},$ $W^{-}$ and $R$ stand for the self-dual, anti-self-dual and scalar curvature of $M^4,$ respectively. The Weitzenb\"ock formula is a powerful ingredient in the theory of canonical metrics on 4-manifolds. It may be used to obtain classification results as well as rule out some possible new examples. In conjunction with Hitchin's theory \cite{besse}, the Weitzenb\"ock formula provides the classification of four-dimensional Einstein manifolds with positive scalar curvature. It was recently shown by Wu \cite{WU} an alternative proof of the classical Weitzenb\"ock formula as well as some classification results for Einstein and conformally Einstein four-dimensional manifolds.  Moreover, he obtained a Weitzenb\"ock type formula for a large class of metrics on four-manifolds which are called generalized $m$-quasi-Einstein metrics.  The proof uses, among other ingredients, an elegant argument of Hamilton \cite{Ham} and Berger curvature decomposition \cite{Berger,besse}.

In the second part of this article, by using a method analogous to the one of Wu \cite{WU}, we shall provide a Weitzenb\"ock type formula for critical metrics of the volume functional on four-dimensional manifolds. More precisely, we have established the following result.

\begin{theorem}
\label{thmW}
Let $(M^4,\,g)$ be a four-dimensional connected, smooth Riemannian manifold and $f$ is a smooth function on $M^4$ satisfying (\ref{eqMiaoTam1}). Then we have:
\begin{eqnarray*}
\div(f^2\nabla|W^\pm|^2)&=&2f^2|\nabla W^\pm|^2+\left(\frac{4Rf^2}{3}+\frac{4f}{3}+8|\nabla f|^2\right)|W^\pm|^2\\
 &&+f\langle \nabla|W^\pm|^2,\nabla f\rangle-144f^2\det W^\pm\\
 &&-4f^2\langle \left(Ric\odot Ric\right)^\pm,W^\pm\rangle+6\langle\left(Ric\odot df\otimes df\right)^\pm,W^\pm\rangle,
\end{eqnarray*} where $\odot$ stands for the Kulkarni-Nomizu product and $\otimes$ is the tensorial product.
\end{theorem}

We point out that Theorem \ref{thmW} holds just assuming that $(M^4,\,g)$ satisfies the equation (\ref{eqMiaoTam1}). In other words, it is not necessary to assume either any boundary condition or compactness.  Next, as an application of the Weitzenb\"ock type formula obtained in Theorem \ref{thmW} we have the following corollary.
\begin{corollary}\label{cor1}
Let $(M^4,\,g,\,f)$  be a four-dimensional simply connected, compact Miao-Tam critical metric with nonnegative scalar curvature and boundary isometric to a standard sphere $\Bbb{S}^3.$ Suppose that

\begin{equation}
\int_{M} \langle \mathring{Ric}\odot df\otimes df, W\rangle dV_{g}\ge \frac{2}{3}\int_{M}\Big(\sqrt{6}|\mathring{Ric}|^{2}+4|W|^{2}\Big) f^{2}|W| dV_{g},
\end{equation} where $\mathring{Ric}$ stands for the  traceless Ricci tensor. Then $M^{4}$ is isometric to a geodesic ball in a simply connected space form $\mathbb{R}^{4}$ or $\mathbb{S}^{4}.$
\end{corollary}

\section{Background and Key Lemmas}
\label{Preliminaries}

In this section we shall establish the standard notation and terminology that we follow throughout this paper. Moreover, we shall present some key lemmas which will be useful for the establishment of the desired results. We start remembering that the fundamental equation of a Miao-Tam critical metric is given by
\begin{equation}
\label{eqfund1} -(\Delta f)g+Hess f-fRic=g.
\end{equation} We also mention that a Riemannian manifold $(M^{n},\,g)$ for which there exists a nontrivial function $f$ satisfying (\ref{eqfund1}) must have constant scalar curvature $R$ (cf. Proposition 2.1 in \cite{CEM} and Theorem 7 in \cite{miaotam}). 

Taking the trace of (\ref{eqfund1}) we arrive at
\begin{equation}
\label{eqtrace} \Delta f +\frac{fR+n}{n-1}=0.
\end{equation} Putting these facts together, we get
\begin{equation}
\label{eqVstaic2}\nabla^2f-fRic=-\frac{Rf+1}{n-1}g.
\end{equation} Furthermore, it is not difficult to verify the following identities
\begin{equation}\label{eqdeltaf2}
\frac{1}{2}\Delta f^{2}+\frac{R}{n-1}f^{2}+\frac{n}{n-1}f=|\nabla f|^{2}
\end{equation}
and
\begin{equation}\label{eqdeltaf3}
\frac{1}{3}\Delta f^{3}+\frac{R}{n-1}f^{3}+\frac{n}{n-1}f^{2}=2f|\nabla f|^{2}.
\end{equation} Also, it easy to check from (\ref{eqtrace}) that
\begin{equation}
\label{IdRicHess} f\mathring{Ric}=\mathring{Hess f},
\end{equation}  where  $\mathring{T}$ stands for the traceless part of $T.$ Throughout the article, the Einstein convention of summing over the repeated indices will be adopted.

For what follows, it is important to remember that the Weyl tensor $W$ is defined by the following decomposition formula
\begin{eqnarray}
\label{weyl}
R_{ijkl}&=&W_{ijkl}+\frac{1}{n-2}\big(R_{ik}g_{jl}+R_{jl}g_{ik}-R_{il}g_{jk}-R_{jk}g_{il}\big) \nonumber\\
 &&-\frac{R}{(n-1)(n-2)}\big(g_{jl}g_{ik}-g_{il}g_{jk}\big),
\end{eqnarray} where $R_{ijkl}$ stands for the Riemann curvature tensor, whereas the Cotton tensor $C$ is given by
\begin{equation}
\label{cotton} \displaystyle{C_{ijk}=\nabla_{i}R_{jk}-\nabla_{j}R_{ik}-\frac{1}{2(n-1)}\big(\nabla_{i}R
g_{jk}-\nabla_{j}R g_{ik}).}
\end{equation} Notice that $C_{ijk}$ is skew-symmetric in the first two indices and trace-free in any two indices. We also remember that  $W \equiv 0$ in dimension three.

Following the notation employed in \cite{BDR},  we recall that
the covariant 3-tensor $T_{ijk}$ is defined by
 \begin{eqnarray}\label{TensorT}
T_{ijk}&=&\frac{n-1}{n-2}(R_{ik}\nabla_{j}f-R_{jk}\nabla_{i}f)+\frac{1}{n-2}(g_{ik}R_{js}\nabla_{s}f-g_{jk}R_{is}\nabla_{s}f)\nonumber\\
&&-\frac{R}{n-2}(g_{ik}\nabla_{j}f-g_{jk}\nabla_{i}f).
\end{eqnarray} This tensor is closely tied to the Cotton tensor, and it played a fundamental role in the previous work \cite{BDR} on classifying Bach-flat critical metrics of the volume func\-tio\-nal. The tensor $T_{ijk}$ is also skew-symmetric in their first
two indices and trace-free in any two indices.

In order to set the stage for the proof to follow let us recall an useful result obtained previously in \cite[Lemma 1]{BDR}. 

\begin{lemma}[\cite{BDR}]
\label{L1}
Let $(M^{n},g)$ be a connected, smooth Riemannian manifold and $f$ is a smooth function on $M^n$ satisfying Eq. (\ref{eqMiaoTam1}). Then we have:
$$f(\nabla_{i}R_{jk}-\nabla_{j}R_{ik})=R_{ijkl}\nabla_{l}f+\frac{R}{n-1}(\nabla_{i}fg_{jk}-\nabla_{j}fg_{ik})-(\nabla_{i}fR_{jk}-\nabla_{j}f R_{ik}).$$
\end{lemma}

At the same time, we remember a result obtained by Hamilton \cite[Lemma 7.2]{Ham} which plays a crucial role in this article. 

\begin{lemma}[\cite{Ham}] 
\label{lemHam}
For any metric $g_{ij}$ the curvature tensor $R_{ijkl}$ satisfies the identity
\begin{eqnarray}
\Delta R_{ijkl}&=&\nabla_i(\nabla_kR_{jl}-\nabla_lR_{jk})-\nabla_j(\nabla_kR_{il}-\nabla_lR_{ik})\nonumber\\
 &&-2Q(R)_{ijkl}+(R_{pjkl}R_{pi}-R_{pikl}R_{pj}),
\end{eqnarray} where $Q(R)_{ijkl}=Z_{ijkl}-Z_{ijlk}-Z_{iljk}+Z_{ikjl}$ and $Z_{ijkl}=R_{ipjq}R_{kplq}.$
\end{lemma}

The Kulkarni-Nomizu product  $\odot,$ which takes two symmetric $(0,2)$-tensors and provides a $(0,4)$-tensor with the same algebraic symmetries of the curvature tensor, is defined by
\begin{eqnarray}
(\alpha \odot \beta)_{ijkl}=\alpha_{ik}\beta_{jl}+\alpha_{jl}\beta_{ik}-\alpha_{il}\beta_{jk}-\alpha_{jk}\beta_{il}.
\end{eqnarray} With these notations we may state our first key lemma.

\begin{lemma}
Let $(M^n,\,g)$ be a connected, smooth Riemannian manifold and $f$ is a smooth function on $M^n$ satisfying (\ref{eqMiaoTam1}). Then we have:
\begin{eqnarray}
\div(f\nabla R_{ijkl})&=&\frac{2Rf+2}{n-1}R_{ijkl}-\left(\nabla^2f\odot\left( Ric-\frac{R}{n-1}g\right)\right)_{ijkl}\nonumber\\
 &&+C_{jil}\nabla_kf+C_{ijk}\nabla_lf+C_{lkj}\nabla_if+C_{kli}\nabla_jf\nonumber\\
 &&-2f Q(R)_{ijkl},
\end{eqnarray} where $C_{ijk}$ is the Cotton tensor.
\end{lemma}

\begin{proof} First of all,  we need to express $f\Delta R_{ijkl}$ in terms of the Cotton tensor. To that end, we first use Lemma \ref{lemHam} to infer
\begin{eqnarray*} 
f\Delta R_{ijkl}&=&f\nabla_i(\nabla_kR_{jl}-\nabla_lR_{jk})-f\nabla_j(\nabla_kR_{il}-\nabla_lR_{ik})\\
 &&+f(R_{pjkl}R_{pi}-R_{pikl}R_{pj})-2fQ(R)_{ijkl}\\
 &=&\nabla_i[f(\nabla_kR_{jl}-\nabla_lR_{jk})]-(\nabla_kR_{jl}-\nabla_lR_{jk})\nabla_if\\
 &&-\nabla_j[f(\nabla_kR_{il}-\nabla_lR_{ik})]+(\nabla_kR_{il}-\nabla_lR_{ik})\nabla_jf\\
 &&+f(R_{pjkl}R_{pi}-R_{pikl}R_{pj})-2fQ(R)_{ijkl},
\end{eqnarray*} and therefore by Lemma \ref{L1} we obtain

\begin{eqnarray*}
f\Delta R_{ijkl}&=&\nabla_i[R_{kljp}\nabla_p f+\frac{R}{n-1}(\nabla_kfg_{jl}-\nabla_lfg_{jk})-(\nabla_kfR_{jl}-\nabla_lfR_{jk})]\\
 &&-\nabla_j[R_{klip}\nabla_p f+\frac{R}{n-1}(\nabla_kfg_{il}-\nabla_lfg_{ik})-(\nabla_kfR_{il}-\nabla_lfR_{ik})]\\
 &&-(\nabla_kR_{jl}-\nabla_lR_{jk})\nabla_if+(\nabla_kR_{il}-\nabla_lR_{ik})\nabla_jf\\
 &&+f(R_{pjkl}R_{pi}-R_{pikl}R_{pj})-2fQ(R)_{ijkl}.
 \end{eqnarray*} Rearranging the terms we get

 \begin{eqnarray*}
 f\Delta R_{ijkl}&=&\nabla_iR_{kljp}\nabla_p f+R_{kljp}\nabla_i\nabla_pf+\frac{R}{n-1}(\nabla_i\nabla_kfg_{jl}-\nabla_i\nabla_lfg_{jk})\\
 &&-(\nabla_i\nabla_k fR_{jl}+\nabla_kf\nabla_iR_{jl}-\nabla_i\nabla_lfR_{jk}-\nabla_lf\nabla_iR_{jk})\\
 &&-\nabla_jR_{klip}\nabla_pf-R_{klip}\nabla_j\nabla_pf-\frac{R}{n-1}(\nabla_j\nabla_kfg_{il}-\nabla_j\nabla_lfg_{ik})\\
 &&+(\nabla_j\nabla_kfR_{il}+\nabla_kf\nabla_jR_{il}-\nabla_j\nabla_lfR_{ik}-\nabla_lf\nabla_jR_{ik})\\
 &&-(\nabla_kR_{jl}-\nabla_lR_{jk})\nabla_if+(\nabla_kR_{il}-\nabla_lR_{ik})\nabla_jf\\
 &&+f(R_{pjkl}R_{pi}-R_{pikl}R_{pj})-2fQ(R)_{ijkl}\\
 &=&\nabla_iR_{jpkl}\nabla_pf+\nabla_jR_{pikl}\nabla_pf+R_{jpkl}\nabla_i\nabla_pf-R_{ipkl}\nabla_j\nabla_pf\\
 &&+\frac{R}{n-1}(\nabla_i\nabla_kfg_{jl}-\nabla_i\nabla_lfg_{jk}-\nabla_j\nabla_kfg_{il}+\nabla_j\nabla_lfg_{ik})\\
 &&-(\nabla_i\nabla_kfR_{jl}-\nabla_i\nabla_lfR_{jk}-\nabla_j\nabla_kfR_{il}+\nabla_j\nabla_lfR_{ik})\\
 &&+(\nabla_jR_{il}-\nabla_iR_{jl})\nabla_kf+(\nabla_iR_{jk}-\nabla_jR_{ik})\nabla_lf\\
 &&-(\nabla_kR_{jl}-\nabla_lR_{jk})\nabla_if+(\nabla_kR_{il}-\nabla_lR_{ik})\nabla_jf\\
 &&+f(R_{pjkl}R_{pi}-R_{pikl}R_{pj})-2fQ(R)_{ijkl}.
  \end{eqnarray*} Thus, since $M^n$ has constant scalar curvature, it follows from (\ref{cotton}) that

  \begin{eqnarray}
  \label{jk2}
  f\Delta R_{ijkl}&=&\nabla_iR_{jpkl}\nabla_pf+\nabla_jR_{pikl}\nabla_pf+R_{jpkl}\nabla_i\nabla_p f-R_{ipkl}\nabla_j\nabla_p f\nonumber\\
 &&-\left(\nabla^2f\odot\left( Ric-\frac{R}{n-1}g\right)\right)_{ijkl}+C_{jil}\nabla_kf+C_{ijk}\nabla_lf+C_{lkj}\nabla_i f\nonumber\\
 &&+C_{kli}\nabla_jf+f(R_{pjkl}R_{pi}-R_{pikl}R_{pj})-2fQ(R)_{ijkl}.
\end{eqnarray} Proceeding, we use the Bianchi identity and Eq. (\ref{eqVstaic2}) to arrive at

\begin{eqnarray*}
f\Delta R_{ijkl}&=&-\nabla_pR_{ijkl}\nabla_pf+R_{jpkl}\nabla_i\nabla_pf-R_{ipkl}\nabla_j\nabla_pf\nonumber\\
 &&-\left(\nabla^2f\odot\left( Ric-\frac{R}{n-1}g\right)\right)_{ijkl}+C_{jil}\nabla_kf+C_{ijk}\nabla_lf+C_{lkj}\nabla_if\nonumber\\
 &&+C_{kli}\nabla_jf+f(R_{pjkl}R_{pi}-R_{pikl}R_{pj})-2fQ(R)_{ijkl}\nonumber\\
 &=&-\nabla_pR_{ijkl}\nabla_pf+R_{jpkl}\Big(fR_{ip}-\frac{Rf+1}{n-1}g_{ip}\Big)-R_{ipkl}\Big(fR_{jp}-\frac{Rf+1}{n-1}g_{jp}\Big)\nonumber\\
 &&-\left(\nabla^2f\odot\left(Ric-\frac{R}{n-1}g\right)\right)_{ijkl}+C_{jil}\nabla_kf+C_{ijk}\nabla_lf+C_{lkj}\nabla_if\nonumber\\
 &&+ C_{kli}\nabla_jf+f(R_{pjkl}R_{pi}-R_{pikl}R_{pj})-2fQ(R)_{ijkl}\\&=&-\nabla_pR_{ijkl}\nabla_pf-f(R_{pjkl}R_{pi}-R_{pikl}R_{pj})+2\frac{(Rf+1)}{n-1}R_{ijkl}\nonumber\\
 &&-\left(\nabla^2f\odot\left( Ric-\frac{R}{n-1}g\right)\right)_{ijkl}+C_{jil}\nabla_kf+C_{ijk}\nabla_lf+C_{lkj}\nabla_if\nonumber\\
 &&+C_{kli}\nabla_jf+f(R_{pjkl}R_{pi}-R_{pikl}R_{pj})-2fQ(R)_{ijkl},
 \end{eqnarray*} which can rewritten succinctly as

  \begin{eqnarray}\label{eqDeltaRm}
f\Delta R_{ijkl}&=&\frac{2Rf+2}{n-1}R_{ijkl}-\left(\nabla^2f\odot\left( Ric-\frac{R}{n-1}g\right)\right)_{ijkl}-\nabla_pR_{ijkl}\nabla_pf\nonumber\\
 &&+C_{jil}\nabla_kf+C_{ijk}\nabla_lf+C_{lkj}\nabla_if+C_{kli}\nabla_jf-2fQ(R)_{ijkl}.
\end{eqnarray}

Finally, we are ready to compute $\div(f\nabla R_{ijkl}).$ To do so, it suffices to observe that $$\div(f\nabla R_{ijkl})=f\Delta R_{ijkl}+\nabla_pR_{ijkl}\nabla_p f,$$ and this combined with (\ref{eqDeltaRm}) yields
\begin{eqnarray*}
\div(f\nabla R_{ijkl})&=&\frac{2Rf+2}{n-1}R_{ijkl}-\left(\nabla^2f\odot\left( Ric-\frac{R}{n-1}g\right)\right)_{ijkl}\\
 &&+C_{jil}\nabla_kf+C_{ijk}\nabla_lf+C_{lkj}\nabla_if+C_{kli}\nabla_jf-2fQ(R)_{ijkl}.
\end{eqnarray*} So, the proof of the lemma is finished.
\end{proof}

To conclude this section we shall provide an useful expression for $\div(f^2\nabla|W|^2),$ which is a key ingredient in the proof of Theorem \ref{thmW}. Before to do that, we mention that in the remainder of this section, we will always consider $$\langle S,T\rangle=S_{ijkl}T^{ijkl},$$ for any $(0,4)$-tensors $S$ and $T.$ Our convention differs from \cite{WU} by $\frac{1}{4}.$

\begin{lemma}\label{Lem2}
Let $(M^n,\,g,\,f)$ be a connected, smooth Riemannian manifold and $f$ is a smooth function on $M^n$ satisfying (\ref{eqMiaoTam1}). Then we have:
\begin{eqnarray}
\label{1ad}
\div(f^2\nabla|W|^2)&=&\frac{4Rf^2}{n-1}|W|^2+\frac{4f}{n-1}|W|^2-\frac{2nf^2}{n-2}\langle Ric\odot Ric,W\rangle\nonumber\\
 &&+f\langle \nabla|W|^2,\nabla f\rangle+\frac{4(n-1)}{n-2}\langle Ric\odot df\otimes df,W\rangle\\
 &&+8|\iota_{\nabla f}W|^2-4f^2\langle Q(W),W\rangle+2f^2|\nabla W|^2\nonumber,
\end{eqnarray} where  $\iota$ is the interior multiplication and $\otimes$ stands for the tensorial product.
\end{lemma}

\begin{proof} Since $W$ is traceless, it is not difficult to see that $\langle A\odot g,W\rangle=0$ for any $(0,2)$-tensor $A.$ This data jointly with (\ref{eqDeltaRm}) gives

\begin{eqnarray*}
f^2\langle\Delta Rm,W\rangle&=&f^2\Delta R_{ijkl}W_{ijkl}\\
 &=&\frac{2Rf^2+2f}{n-1}R_{ijkl}W_{ijkl}-f\left[\nabla^2f\odot\left(Ric-\frac{R}{n-1}g\right)\right]_{ijkl}W_{ijkl}\\
 &&-f\nabla_pR_{ijkl}\nabla_pfW_{ijkl}+fC_{ijk}\nabla_lfW_{ijkl}+fC_{jil}\nabla_kfW_{ijkl}\\
 &&+fC_{kli}\nabla_jfW_{ijkl}+fC_{lkj}\nabla_ifW_{ijkl}-2f^2Q(R)_{ijkl}W_{ijkl}.
 \end{eqnarray*} Consequently, it follows from (\ref{weyl}) that

 \begin{eqnarray*}
 f^2\langle\Delta Rm,W\rangle&=&\frac{2Rf^2}{n-1}|W|^2+\frac{2f}{n-1}|W|^2-f(\nabla^2f\odot Ric)_{ijkl}W_{ijkl}\\
 &&-f\nabla_pR_{ijkl}\nabla_pfW_{ijkl}+4fC_{ijk}\nabla_lfW_{ijkl}-2f^2Q(R)_{ijkl}W_{ijkl}
\end{eqnarray*} and hence, using Eq. (1.3) of \cite{CM} and once more (\ref{weyl}), we obtain
\begin{eqnarray*}
f^2\langle\Delta Rm,W\rangle&=&\frac{2Rf^2}{n-1}|W|^2+\frac{2f}{n-1}|W|^2-f(\nabla^2f\odot Ric)_{ijkl}W_{ijkl}\\
 &&-f\nabla_pW_{ijkl}\nabla_pfW_{ijkl}+4fC_{ijk}\nabla_lfW_{ijkl}-2f^2Q(W)_{ijkl}W_{ijkl}\\
 &&-\frac{4f^2}{n-2}(R_{ik}R_{jl}-R_{il}R_{jk})W_{ijkl}.
\end{eqnarray*} Thereby, from Lemma 2 of \cite{BDR} and (\ref{eqfund1}) we achieve

\begin{eqnarray}
\label{ei12}
f^2\langle\Delta Rm,W\rangle &=&\frac{2Rf^2}{n-1}|W|^2+\frac{2f}{n-1}|W|^2-f(\nabla^2f\odot Ric)_{ijkl}W_{ijkl}\nonumber\\
 &&-f\nabla_p W_{ijkl}\nabla_pfW_{ijkl}+4T_{ijk}\nabla_lfW_{ijkl}+4|\iota_{\nabla f}W|^2\nonumber\\
 &&-2f^2\langle Q(W),W\rangle-\frac{2f^2}{n-2}(Ric\odot Ric)_{ijkl}W_{ijkl},
 \end{eqnarray} where  $\iota$ is the {\it interior multiplication} and $T_{ijk}$ was defined in (\ref{TensorT}). In particular, the expression of (\ref{TensorT}) substituted into (\ref{ei12}) jointly with (\ref{eqMiaoTam1}) immediately arrives at

 \begin{eqnarray}
   \label{pl1}
f^2\langle\Delta Rm,W\rangle &=&\frac{2Rf^2}{n-1}|W|^2+\frac{2f}{n-1}|W|^2-\frac{nf^2}{n-2}(Ric\odot Ric)_{ijkl}W_{ijkl}\nonumber\\
 &&-f\nabla_pW_{ijkl}\nabla_pfW_{ijkl}+\frac{4(n-1)}{n-2}(R_{ik}\nabla_jf\nabla_lf-R_{jk}\nabla_if\nabla_lf)W_{ijkl}\nonumber\\
 &&+4|\iota_{\nabla f}W|^2-2f^2\langle Q(W),W\rangle.
\end{eqnarray} Easily one verifies that $$(Ric\odot df\otimes df)_{ijkl}W_{ijkl}=2(R_{ik}\nabla_jf\nabla_lf-R_{jk}\nabla_if\nabla_lf)W_{ijkl},$$ and plugging this in (\ref{pl1}) we see that
\begin{eqnarray}
\label{jkh}
f^2\langle\Delta Rm,W\rangle&=&\frac{2Rf^2}{n-1}|W|^2+\frac{2f}{n-1}|W|^2-\frac{nf^2}{n-2}\langle Ric\odot Ric,W\rangle\nonumber\\
 &&-\frac{1}{2}f\nabla_p|W|^2\nabla_pf+\frac{2(n-1)}{n-2}(Ric\odot df\otimes df)_{ijkl}W_{ijkl}\nonumber\\
 &&+4|\iota_{\nabla f}W|^2-2f^2\langle Q(W),W\rangle\nonumber\\
 &=&\frac{2Rf^2}{n-1}|W|^2+\frac{2f}{n-1}|W|^2-\frac{nf^2}{n-2}\langle Ric\odot Ric,W\rangle\nonumber\\
 &&-\frac{1}{2}f\langle \nabla|W|^2,\nabla f\rangle+\frac{2(n-1)}{n-2}\langle Ric\odot df\otimes df,W\rangle\\
 &&+4|\iota_{\nabla f}W|^2-2f^2\langle Q(W),W\rangle\nonumber.
\end{eqnarray}

On the other hand, it is straightforward to verify that

\begin{eqnarray}
\label{lkj1}
\div(f^2\nabla|W|^2)&=&f^2\Delta|W|^2+\langle\nabla|W|^2,\nabla f^2\rangle\nonumber\\
&=&2f^2\langle\Delta W,W\rangle+2f^2|\nabla W|^2+2f\langle\nabla|W|^2,\nabla f\rangle\nonumber\\
 &=&2f^2\langle\Delta Rm,W\rangle+2f^2|\nabla W|^2+2f\langle\nabla|W|^2,\nabla f\rangle.
 \end{eqnarray} Together, (\ref{jkh}) and (\ref{lkj1}) yield

\begin{eqnarray*}
\div(f^2\nabla|W|^2) &=&\frac{4Rf^2}{n-1}|W|^2+\frac{4f}{n-1}|W|^2-\frac{2nf^2}{n-2}\langle Ric\odot Ric,W\rangle\\
 &&+f\langle \nabla|W|^2,\nabla f\rangle+\frac{4(n-1)}{n-2}\langle Ric\odot df\otimes df,W\rangle\\
 &&+8|\iota_{\nabla f}W|^2 -4f^2\langle Q(W),W\rangle+2f^2|\nabla W|^2,
\end{eqnarray*} which gives the desired result.
\end{proof}

As already mentioned, 4-manifolds display peculiar features. For ins\-tan\-ce, the bundle of $2$-forms on a four-dimensional oriented Riemannian manifold can be invariantly decomposed as a direct sum $$\Lambda^2=\Lambda^{+}\oplus\Lambda^{-}.$$ This decomposition is conformally invariant. Moreover, it allows us to conclude that the Weyl tensor $W$ is an endomorphism of $\Lambda^2=\Lambda^{+} \oplus \Lambda^{-} $ such that $W = W^+\oplus W^-,$ where $W^{+}$ and $W^{-}$ stand for the self-dual and anti-self-dual parts of the Weyl tensor of $M^4,$ respectively. In particular, we also mention that a four-dimensional connected, smooth Riemannian manifold $(M^4,\,g)$ and a smooth function $f$ on $M^4$ satisfying (\ref{eqMiaoTam1}) must satisfy

\begin{eqnarray}
\label{keq}
\div(f^2\nabla|W^{\pm}|^2)&=&\frac{4Rf^2}{3}|W^{\pm}|^2+\frac{4f}{3}|W^{\pm}|^2- 4f^2\langle (Ric\odot Ric)^{\pm},W^{\pm}\rangle\nonumber\\
 &&+f\langle \nabla|W^{\pm}|^2,\nabla f\rangle+6\langle (Ric\odot df\otimes df)^{\pm},W^{\pm}\rangle\\
 &&+8|\iota_{\nabla f}W^{\pm}|^2-4f^2\langle Q(W)^{\pm},W^{\pm}\rangle+2f^2|\nabla W^{\pm}|^2\nonumber.
\end{eqnarray}

It should be pointed out that the proof of (\ref{keq}) for the anti-self-dual and self-dual parts of the Weyl tensor can be implemented in quite the same way of the proof of Lemma \ref{Lem2}, so it is omitted.

\section{Proof of the Main Results}

\subsection{Proof of Theorem \ref{thmMainA}}
\begin{proof} To begin with, notice that the boundary condition and (\ref{eqMiaoTam1}) implies

$$D_{X}|\nabla f|^{2}=2Hess\,f(X,\nabla f)=0,$$ for all $X$ tangent to $\partial M.$ In particular, it is easy to see that $|\nabla f|\neq 0$ along $\partial M.$

We now recall the classical B\"ochner formula (cf. \cite{chavel}, p. 83):

$$\frac{1}{2}\Delta|\nabla f|^{2}=Ric(\nabla f, \nabla f)+\langle\nabla\Delta f,\nabla f\rangle+|Hess\, f|^{2}.$$ This jointly with (\ref{eqtrace}) yields
\begin{eqnarray*}
\int_{M}f\Delta |\nabla f|^{2}dV_{g}&=&2\int_{M}f Ric(\nabla f,\nabla f) dV_{g}-\frac{2R}{n-1}\int_{M}f|\nabla f|^{2} dV_{g}\\
&&+2\int_{M}f|Hess\, f|^{2}dV_{g}.
\end{eqnarray*} With aid of (\ref{IdRicHess}) we can rewtite this above equation as
\begin{eqnarray}\label{intaux1}
\int_{M}f\Delta |\nabla f|^{2}dV_{g}&=&2\int_{M}f Ric(\nabla f,\nabla f) dV_{g}-\frac{2R}{n-1}\int_{M}f|\nabla f|^{2} dV_{g}\nonumber\\
&&+2\int_{M}f^{3}|\mathring{Ric}|^{2}dV_{g}+\frac{2}{n}\int_{M}f(\Delta f)^{2}dV_{g}.
\end{eqnarray}

Proceeding, it is not hard to check that 

\begin{eqnarray*}
2\int_{M}fRic(\nabla f,\nabla f) dV_{g}&=&\int_{M}{\rm div}(f^{2}Ric(\nabla f))dV_{g}-\int_{M}f^{2}R_{ij}\nabla_{i}\nabla_{j}f dV_{g},
\end{eqnarray*} where we have used the twice contracted second Bianchi identity ($2{\rm div}Ric=\nabla R=0$). Next, we use (\ref{eqfund1}) and that $\mathring{Ric}=Ric-\frac{R}{n}g$ to get

\begin{eqnarray}\label{intRicff}
2\int_{M}fRic(\nabla f,\nabla f) dV_{g}&=&-\int_{M}f^{3}|\mathring{Ric}|^{2}dV_{g}+\frac{1}{n(n-1)}\int_{M}R^{2}f^{3}dV_{g}\nonumber\\
&&+\frac{1}{n-1}\int_{M}Rf^{2}dV_{g}.
\end{eqnarray} 

On the other hand, we use (\ref{eqtrace}) to arrive at
\begin{eqnarray*}
{\rm div}(f\nabla|\nabla f|^{2}-|\nabla f|^{2}\nabla f)&=&f\Delta|\nabla f|^{2}+\frac{R}{n-1}f|\nabla f|^{2}+\frac{n}{n-1}|\nabla f|^{2}.
\end{eqnarray*} Upon integrating this expression over $M^n$ we use Stokes's formula to obtain
\begin{eqnarray}\label{intLf}
\int_{M}f\Delta |\nabla f|^{2}dV_{g}&=&|\nabla f|^{3}|\partial M|-\frac{R}{n-1}\int_{M}f|\nabla f|^{2}dV_{g}\nonumber\\
&&-\frac{n}{n-1}\int_{M}|\nabla f|^{2} dV_{g}.
\end{eqnarray} Substituting (\ref{intRicff}) into (\ref{intaux1}) and comparing with (\ref{intLf}) we obtain
\begin{eqnarray*}
|\nabla f|^{3}|\partial M|&=&\int_{M}f^{3}|\mathring{Ric}|^{2}dV_{g}-\frac{1}{n-1}\int_{M}Rf|\nabla f|^{2}dV_{g}\\
&&+\frac{1}{n(n-1)}\int_{M}R^{2}f^{3}dV_{g}+\frac{1}{n-1}\int_{M}Rf^{2}dV_{g}\\
&&+\frac{2}{n}\int_{M}f(\Delta f)^{2}dV_{g}+\frac{n}{n-1}\int_{M}|\nabla f|^{2}dV_{g}.
\end{eqnarray*} Next, use (\ref{eqtrace}) and (\ref{eqdeltaf3}) to arrive at

\begin{eqnarray*}
|\nabla f|^{3}|\partial M|&=&\int_{M}f^{3}|\mathring{Ric}|^{2}dV_{g}+\frac{n-2}{2n(n-1)^{2}}\int_{M}R^{2}f^{3}dV_{g}\\
&&+\frac{n-2}{2(n-1)^{2}}\int_{M}Rf^{2}dV_{g}+\frac{2}{n}\int_{M}f\left(\frac{Rf}{n-1}+\frac{n}{n-1}\right)^{2}dV_{g}\\
&&+\frac{n}{n-1}\int_{M}|\nabla f|^{2}dV_{g}\\
&=&\int_{M}f^{3}|\mathring{Ric}|^{2}dV_{g}+\frac{n+2}{2n(n-1)^{2}}\int_{M}R^{2}f^{3}dV_{g}\\
&&+\frac{n+6}{2(n-1)^{2}}\int_{M}Rf^{2}dV_{g}+\frac{2n}{(n-1)^{2}}\int_{M}f dV_{g}\\
&&+\frac{n}{n-1}\int_{M}|\nabla f|^{2}dV_{g}.
\end{eqnarray*}

From this it follows that
\begin{eqnarray}
\label{eqboM0}
|\partial M|&=& H^{3}\int_{M}f^{3}|\mathring{Ric}|^{2}dV_{g}+\frac{n+2}{2n(n-1)^{2}}H^{3}\int_{M}R^{2}f^{3}dV_{g}\nonumber\\
&&+\frac{3(n+2)}{2(n-1)^{2}}H^{3}\int_{M}Rf^{2}dV_{g}+\frac{n(n+2)}{(n-1)^{2}}H^{3}\int_{M}fdV_{g}\nonumber\\
&=& H^{3}\int_{M}f^{3}|\mathring{Ric}|^{2}dV_{g}\nonumber\\&&+\frac{n+2}{2n(n-1)^{2}}H^{3}\int_{M}[R^{2}f^{3}+3nRf^{2}+2n^{2}f]dV_{g},
\end{eqnarray} which implies
\begin{eqnarray}\label{eqbordoM}
|\partial M|&\geq&\frac{n+2}{2n(n-1)^{2}}H^{3}C_{R},
\end{eqnarray} where $C_{R}$ is a constant given by
$$C_{R}=\int_{M}\big(R^{2}f^{3}+3nRf^{2}+2n^{2}f\big)dV_{g}.$$ Moreover, to see that $C_{R}$ is a positive constant it suffices to use (\ref{eqdeltaf2}) in order to rewrite the last expression as
\begin{eqnarray*}
C_{R}&=&\int_{M}(Rf+n)^{2}fdV_{g}+n(n-1)\int_{M}|\nabla f|^{2} dV_{g},
\end{eqnarray*} which is clearly positive. 

Finally, from (\ref{eqboM0}) we deduce that the equality holds in (\ref{eqbordoM}) if and only if $M^{n}$ is Einstein. Hence, we may apply Theorem 1.1 of  \cite{miaotamTAMS} to conclude that $M^n$ is isometric to a geodesic ball in a simply connected space form $\Bbb{R}^{n},$ $\Bbb{H}^{n}$ or $\Bbb{S}^{n}.$ The proof is completed.
\end{proof}

\subsection{Proof of Theorem~\ref{thmMainB}}

\begin{proof} Firstly, it is easy to verify from (\ref{eqtrace}) that
$$|\nabla f||\partial M|=\frac{R}{n-1}\int_{M}fdV_{g}+\frac{n}{n-1}Vol(M).$$ Since $M^n$ has nonnegative scalar curvature, we have
\begin{equation}\label{auxbordoM}
|\partial M|\geq\frac{nH}{n-1}Vol(M),
\end{equation} where $H$ is the mean curvature of $\partial M$ with respect to the outward unit normal. We then combine Theorem \ref{thmMainA} with (\ref{auxbordoM}) to infer
\begin{eqnarray}\label{auxB}
|\partial M|^{n}&\geq&\frac{n+2}{2n(n-1)^{2}}H^{3}C_{R}|\partial M|^{n-1}\nonumber\\
&\geq&\frac{n+2}{2n(n-1)^{2}}H^{3}C_{R}\left(\frac{nH}{n-1}Vol(M)\right)^{n-1}\nonumber\\
&=&C_{R,H} Vol(M)^{n-1},
\end{eqnarray} where $C_{R,H}$ is a constant given by
$$C_{R,H}=\frac{(n+2)n^{n-2}}{2(n-1)^{n+1}}H^{n+2}C_{R},$$
and $C_{R}$ is defined in Theorem \ref{thmMainA}. In particular, the equality holds in (\ref{auxB}) if and only if the equality holds in (\ref{auxbordoM}) as well as in Theorem \ref{thmMainA}. From this it follows that the scalar curvature must be zero and $(M^{n},g)$ is isometric to a geodesic ball in $\mathbb{R}^{n}.$ This is what we wanted to prove.
\end{proof}

\subsection{Proof of Theorem \ref{thmW}}

\begin{proof} From the Berger curvature decomposition (cf. \cite{WU}, p. 1090) we have
\begin{equation}\label{eqdetWeyl}
Q(W)_{ijkl}^\pm W_{ijkl}=36\det W^\pm.
\end{equation} Moreover, from Lemma 3.2 in \cite{WU} we get
\begin{equation}\label{eqContWeyl}
|\iota_{\nabla f}W^\pm|^2=|W^\pm|^2|\nabla f|^2,
\end{equation} and using this together with (\ref{eqdetWeyl}) in (\ref{keq}) gives
\begin{eqnarray*}
\div(f^2\nabla|W^\pm|^2)&=&\frac{4Rf^2}{3}|W^\pm|^2+\frac{4f}{3}|W^\pm|^2-4f^{2}\langle\left(Ric\odot Ric\right)^\pm,W^\pm\rangle\\
 &&+f\langle \nabla|W^\pm|^2,\nabla f\rangle+6\langle\left(Ric\odot df\otimes df\right)^\pm,W^\pm\rangle\\
 &&+8|\nabla f|^2|W^\pm|^2-144f^2\det W^\pm+2f^2|\nabla W^\pm|^2\nonumber,
 \end{eqnarray*} from which we see the proof is completed.
\end{proof}

\subsection{Proof of Corollary \ref{cor1}}
\begin{proof} Upon integrating the expression obtained in Lemma \ref{Lem2} over $M^{4},$ we use the Stokes's formula and that the Weyl tensor $W$ is trace-free to obtain

\begin{eqnarray}
\label{eq1z}
0&=&\int_{M}\Big(\frac{4}{3}Rf^{2}+\frac{4}{3}f\Big)|W|^{2}dV_{g}-4\int_{M}f^{2}\langle \mathring{Ric}\odot \mathring{Ric},W\rangle dV_{g}\nonumber\\&& +\int_{M}f\langle \nabla |W|^{2},\nabla f\rangle dV_{g} +6\int_{M} \langle \mathring{Ric}\odot df\otimes df, W\rangle dV_{g}+8\int_{M}|\iota_{\nabla f}W|^{2}dV_{g}\nonumber\\&& -4\int_{M}f^{2}\langle Q(W), W\rangle dV_{g}+2\int_{M}f^{2}|\nabla W|^{2} dV_{g}.
\end{eqnarray}

On the other hand, we may use (\ref{eqtrace}) to deduce

\begin{eqnarray*}
{\rm div} \Big(f |W|^{2}\nabla f\Big)&=&f |W|^{2}\Delta f + |W|^{2}|\nabla f|^{2}+f\langle \nabla |W|^{2},\nabla f\rangle \nonumber\\ &=& -\frac{R}{3}f^{2}|W|^{2}-\frac{4}{3}f|W|^{2}+|W|^{2}|\nabla f|^{2}+f\langle \nabla |W|^{2}, \nabla f\rangle.
\end{eqnarray*} Now, on integrating this expression over $M^n$ we get

\begin{equation}
\label{eq2z}
\int_{M}f\langle \nabla |W|^{2}, \nabla f\rangle dV_{g}=\int_{M}\Big(\frac{R}{3}f^{2}+\frac{4}{3}f-|\nabla f|^{2}\Big)|W|^{2}dV_{g}.
\end{equation} Besides, it is not difficult to verify that $$|\ric\odot\ric|^2=8|\ric|^4-8|\ric^2|^2,$$ where $(\ric)_{ij}^2=(\ric)_{ik}(\ric)_{kj};$ for more details see \cite{catinoAdv}. In particular, taking into account that $tr \ric^2=|\ric |^{2},$ we immediately have $|\ric^2|^{2}\geq \frac{|\ric|^{4}}{4},$ which implies

\begin{equation}
\label{eq3z}
|\ric\odot\ric|^2\leq 6|\ric|^4.
\end{equation} Returning to Eq. (\ref{eq1z}), we may use (\ref{eq2z}) and (\ref{eq3z}) to deduce

\begin{eqnarray}
\label{eq4z}
0&\geq&\int_{M}\Big(\frac{5}{3}Rf^{2}+\frac{8}{3}f-|\nabla f|^{2}\Big)|W|^{2}dV_{g}-4\sqrt{6}\int_{M}f^{2}| \mathring{Ric}|^{2}|W| dV_{g}\nonumber\\&& +6\int_{M} \langle \mathring{Ric}\odot df\otimes df, W\rangle dV_{g}+8\int_{M}|\iota_{\nabla f}W|^{2}dV_{g}\nonumber\\&&-4\int_{M}f^{2}\langle Q(W), W\rangle dV_{g}+2\int_{M}f^{2}|\nabla W|^{2} dV_{g}\nonumber\\&&\int_{M}\Big(\frac{5}{3}Rf^{2}+\frac{8}{3}f+7|\nabla f|^{2}\Big)|W|^{2}dV_{g}-4\sqrt{6}\int_{M}f^{2}| \mathring{Ric}|^{2}|W| dV_{g}\nonumber\\&& +6\int_{M} \langle \mathring{Ric}\odot df\otimes df, W\rangle dV_{g}+2\int_{M}f^{2}|\nabla W|^{2} dV_{g}\nonumber\\&&-16\int_{M}f^{2}|W|^{3}dV_{g},
\end{eqnarray} where we used the elementary inequality  $|\langle Q(W),W\rangle|\leq 4|W|^{3},$ along with the identity $|\iota_{\nabla f}W|^{2}=|W|^{2}|\nabla f|^{2};$ cf. Lemma 3.2 in \cite{WU}.

In order to conclude it therefore suffices to use our assumption to infer

\begin{equation}
\label{klmn}
\int_{M}\Big(\frac{5}{3}Rf^{2}+\frac{8}{3}f+7|\nabla f|^{2}\Big)|W|^{2}dV_{g}=0.
\end{equation} Hence, since $f$ and $g$ are analytic (cf. \cite[Proposition 2.1]{CEM}), Eq. (\ref{klmn}) forces $M^4$ to be locally conformally flat, and we are in position to use Theorem 1.2 of \cite{miaotamTAMS} to conclude that $M^4$ is isometric to a geodesic ball in a simply connected space form $\mathbb{R}^{4}$ or $\mathbb{S}^{4}.$ This completes the proof of Corollary \ref{cor1}.

\end{proof}

\begin{acknowledgement}
The authors want to thank the referee for his careful reading, relevant remarks and valuable suggestions. Moreover, the authors want to thank R. Batista and P. Wu for helpful conversations about this subject.
\end{acknowledgement}


\begin{thebibliography}{BB}

\bibitem{Lucas} Ambrozio, L.: On static three-manifolds with positive scalar curvature. {\it J. Diff. Geom.} 7 (2017) 1-45.

\bibitem{BalRi2} Baltazar, H. and Ribeiro Jr., E.: Remarks on critical metrics of the scalar curvature and vo\-lu\-me functionals on compact manifolds with boundary. {\it Pacific J. Math.} 297 (2018) 29-45.

\bibitem{baltazar} Baltazar, H. and Ribeiro Jr., E.: Critical metrics of the volume functional on manifolds with boundary. {\it Proc. Amer. Math. Soc.} 145 (2017) 3513-3523.

\bibitem{BLF} Barbosa, E., Lima, L. and Freitas, A.: The generalized Pohozaev-Schoen identity and some geometric applications. To appear in {\it Commun. Anal. Geom.} arXiv:1607.03073v1 [math.DG].

\bibitem{BDR} Barros, A., Di\'ogenes, R. and Ribeiro Jr., E.: Bach-Flat critical metrics of the volume functional on 4-dimensional manifolds with boundary. {\it J. Geom. Anal.} 25 (2015) 2698-2715.

\bibitem{BS} Barros, A. and Da Silva, A.: Rigidity for critical metrics of the volume functional. To appear in {\it Math. Nachr.} arXiv:1706.07367 [math.DG] (2017).

\bibitem{BDRR} Batista, R., Di\'ogenes, R., Ranieri, M. and Ribeiro Jr., E.: Critical metrics of the volume functional on compact three-manifolds with smooth boundary. {\it J. Geom. Anal.} 27 (2017) 1530-1547.

\bibitem{Berger} Berger, M.: Sur quelques vari\'et\'es d’Einstein compactes. {\it Ann. Mat. Pura Appl.}  53 (1961) 89-95.

\bibitem{besse} Besse, A.: Einstein manifolds. Springer-Verlag, Berlin Heidelberg (1987).

\bibitem{BGH} Boucher, W., Gibbons, G. and Horowitz, G.: Uniqueness theorem for anti-de Sitter spacetime. {\it Phys. Rev. D} (3) 30 (1984) 2447-2451.

\bibitem{catinoAdv} Catino, G.: Integral pinched shrinking Ricci solitons. {\it Adv. Math.} 303 (2016) 279-294.

\bibitem{CM} Catino, G. and Mantegazza, C.: The evolution of the Weyl tensor under the Ricci flow. {\it Ann. Inst. Fourier}. 61 (2012) 1407-1435.

\bibitem{chavel} Chavel, I.: Eigenvalues in Riemannian geometry. Pure appl. math., v. 115, Academic Press, New York, 1984.

\bibitem{CEM} Corvino, J., Eichmair, M. and Miao, P.: Deformation of scalar curvature and volume. {\it Math. Annalen.} 357 (2013) 551-584.

\bibitem{derd1} Derdzinski, A.: Self-dual K\"ahler manifolds and Einstein manifolds of dimension four. {\it Compositio Math}. 49 (1983) 405-433.

\bibitem{GS} Gompf, R. and Stipsicz, A.: 4-manifolds and Kirby calculus. Graduate Studies in Mathematics, vol. 20, American Mathematical Society (1999).

\bibitem{Ham} Hamilton, R.:  Three-manifolds with positive Ricci curvature. {\it J. Diff. Geom.} 17 (1982) 255-306.

\bibitem{Kim} Kim, J. and Shin, J.: Four-dimensional static and related critical spaces with harmonic curvature. {\it Pacific J. Math.} 295 (2018) 429-462.

\bibitem{Ma} Maz’ya. M.: Lectures on Isoperimetric and Isocapacitary Inequalities in the Theory of Sobolev Spaces. {\it Contemp. Math.} 338 (2003) 307-340.

\bibitem{miaotam} Miao, P. and  Tam, L.-F.: On the volume functional of compact manifolds with boundary with constant scalar curvature. {\it Calc. Var. PDE.} 36 (2009) 141-171.

\bibitem{miaotamTAMS} Miao, P. and  Tam, L.-F.: Einstein and conformally flat critical metrics of the volume functional. {\it Trans. Amer. Math. Soc.} 363 (2011) 2907-2937.

\bibitem{Ros} Ros, A.: The isoperimetric problem. {\it Global theory of minimal surfaces}. 2 (2001) 175-209.

\bibitem{SY} Schoen, R. and Yau S.-T.: Lectures on Differential Geometry, Conference Proceedings and Lecture Notes in Geometry and Topology, I. International Press, Cambridge,  MA, 1994.

\bibitem{Shen} Shen, Y.:  A note on Fischer-Marsden’s conjecture. {\it Proc. Amer. Math. Soc.} 125 (1997) 901-905.

\bibitem{WU} Wu, P.: A Weitzenbock formula for canonical metrics on four-manifolds. {\it Trans. Amer. Math. Soc.} 369 (2017) 1079-1096.

\bibitem{yuan} Yuan, W.: Volume comparison with respect to scalar curvature. arXiv:1609.08849v1 [math.DG] (2016).


\end{thebibliography}
\end{document}